\documentclass[10pt]{article}
\usepackage{amsmath, amsthm, amssymb, amsfonts, mathrsfs, mathtools}
\usepackage{graphicx}
\usepackage{makeidx}
\usepackage[left=2cm,right=2cm,top=2cm,bottom=2cm]{geometry}
\usepackage{caption}
\usepackage{subcaption}

\usepackage[section]{placeins}

\usepackage{tikz}
\usetikzlibrary{decorations.pathreplacing}
\tikzstyle{vertex}=[auto=left,circle,draw=black,fill=white, inner sep=1.5]

\newtheorem{theorem}{Theorem}[section]

\newtheorem{prop}[theorem]{Proposition}
\newtheorem{corollary}{Corollary}[theorem]
\newtheorem{lemm}{Lemma}[section]

\title{ Signed Complete 
Graphs on Six Vertices}

\author{ Deepak\\
Department of Mathematics\\
Indian Institute of Technology Guwahati\\
Guwahati, India - 781039\\
Email: deepakmath55555@iitg.ac.in\\
\\
Bikash Bhattacharjya\\
Department of Mathematics\\
Indian Institute of Technology Guwahati\\
Guwahati, India - 781039\\
Email: b.bikash@iitg.ac.in
}

\begin{document}
\maketitle

\vspace{-0.3in}

\begin{center}{Abstract}\end{center}
A signed graph is a graph whose edges are labeled positive or negative. The sign of a cycle is the product of the signs of its edges. Zaslavsky proved in 2012 that, up to switching isomorphism, there are six different signed Petersen graphs. It is also known that, up to switching isomorphism, there are two signed $K_3$'s, three signed $K_4$'s, and seven signed $K_5$'s. In this paper, we prove that there are sixteen signed $K_6$'s upto switching ismomorphism.

 
\section{Introduction}
Throughout the paper we only consider simple graphs. For all the graph-theoretic terms that have not been defined but are used in the paper, see Bondy \cite{Bondy}. Harary \cite{F.Harary} was the first to introduced signed graph and balance. Harary \cite{Harary} used signed graphs to model social stress in small groups of people in social psychology. Subsequently, signed graphs have turned out to be valuable. The fundamental property of signed graphs is balance. A cycle is positive or negative according as the product of the signs of its edges is positive or negative. A signed graph is balanced if all its cycles are positive. The second basic property of signed graphs is switching equivalence. Switching is a way of turning one signature of a graph into another, without changing the sign of its cycles. Many properties of signed graphs are unaltered by switching, the set of unbalanced cycles is a notable example. The author in \cite{T.Zaslavsky} described the non-isomorphic signed Petersen graph. In \cite{Sivaraman}, the author have studied the non-isomorphic signatures on Heawood graph. In this paper we determine the non-isomorphic signatures on $K_6$. 

\section{Preliminaries}
A \textit{signified graph} is a graph $G$ together with an assignment of $+$ or $-$ signs to its edges. If $\Sigma$ is the set of negative edges of a graph $G$, then we denote the signified graph by $(G, \Sigma)$. The set $\Sigma$ is called the signature of $(G, \Sigma)$. Signature $\Sigma$ of graph $G$ can also  be viewed as a function from $E(G)$ to $\{+1, -1\}$. A \textit{switching (resigning)} of a signified graph at a vertex $v$ is to change the sign of each edge incident to $v$. We say $(G,\Sigma_{2})$ is \textit{switching equivalent} or simply \textit{equivalent} to $(G,\Sigma_{1})$ if it is obtained from $(G,\Sigma_{1})$ by a sequence of switchings. Equivalently, we say $(G,\Sigma_{2})$ is switching equivalent to $(G,\Sigma_{1})$ if there exist a function $f$ from $V$ to $\{+1, -1\}$ such that $\Sigma_{2}(e)$ = $f(u)\Sigma_{1}(e)f(v)$ for each edge $e = uv$. Switching defines an equivalence relation on the set of all signified graphs over $G$ (also on the set of signatures). Each such equivalence class is called a \textit{signed graph} and is denoted by $[G,\Sigma]$, where $(G,\Sigma)$ is any member of the class.

We say that two signified graphs $(G, \Sigma_{1})$ and $(H, \Sigma_{2})$ are \textit{isomorphic}, denoted $\Sigma_{1}$ $\cong$ $\Sigma_{2}$, if there exist a graph isomorphism $\psi : V(G) \rightarrow V(H)$ which preserves the edge signs. Signified graphs $(G, \Sigma_{1})$ and $(H, \Sigma_{2})$ are called switching isomorphic if $\Sigma_{1}$ is isomorphic to a switching of $\Sigma_{2}$. That is, there exist a signified graph $(H, \Sigma_{2}^{\prime})$ which is equivalent to $(H, \Sigma_{2})$ such that $\Sigma_{1}$ $\cong$ $\Sigma_{2}^{\prime}$. We use the notation $\Sigma_{1}$ $\sim$ $\Sigma_{2}$ to mean that $\Sigma_{1}$ is switching isomorphic to $\Sigma_{2}$.

\begin{prop} \cite{Naserasr} \label{Distinct SG}
  If $G$ has $m$ edges, $n$ vertices and $c$ components, then there are $2^{(m-n+c)}$ distinct signed graphs on $G$.
\end{prop}

One of the first theorems in the theory of signed graphs is that the set of unbalanced cycles uniquely determines the class of signed graphs to which a signified graph belongs. More precisely, we state the following theorem.

\begin{theorem}\cite{Zaslavsky}\label{Signature}
Two signatures $\Sigma_{1}$ and $\Sigma_{2}$ are equivalent if and only if they have the same set of unbalanced cycles.
\end{theorem}

Thus if the signed graphs $[G, \Sigma_1]$ and $[G, \Sigma_2]$ have different sets of negative cycles, then they cannot be switching isomorphic. 
 
\section{Notations} 
In a signed graph $[G, \Sigma]$, a signature $\Sigma^{\prime}$ which is equivalent to $\Sigma$ is said to be a \textit{minimal signature} if the number of edges in $\Sigma^{\prime}$ is minimum among all equivalent signatures of $\Sigma$. We denote the number of edges in $\Sigma^{\prime}$ by $|\Sigma^{\prime}|$.

 For example, if $[G, \Sigma]$ is balanced then $\Sigma^{\prime}$ = $\emptyset$ and so $|\Sigma^{\prime}| = 0$. Notice that there may be two or more than two minimal signatures for a signed graph $[G, \Sigma]$. For example, in the signed graph $[K_3, \Sigma]$, where $\Sigma$ = $\{12, 23, 31\}$, the equivalent signatures $\Sigma_1 = \{12\}$ and $\Sigma_2 = \{23\}$ are minimal signatures. This shows that minimal signature of a signed graph is not unique. 

Note that automorphism group of $K_6$ is $S_6$, and it is vertex-transitive as well as edge-transitive. Thus it is easy to see that for any two isomorphic subgraphs of $K_6$, there exist an automorphism of $K_6$ which maps one subgraph onto the other subgraph.

 The notation $\Sigma(e_1,e_2,...,e_k)$ denotes a signature $\Sigma$ which contains the edges $e_1, e_2,...,e_k$ of a graph $G$. For example, in the graph $K_6$ of Figure~\ref{1}, $\Sigma(u_1u_2, u_3u_4)$ denotes a signature containing the edges $u_1u_2$ and $u_3u_4$.

Further, we say that two signatures $\Sigma_1$ and $\Sigma_2$ of a graph $G$ are automorphic if there exists an automorphism $f$ of $G$ such that $uv \in \Sigma_1$ if and only if $f(u)f(v) \in \Sigma_2$. If two signatures are automorphic then they are said to be \textit{automorphic type} signatures. If two signatures $\Sigma_1$ and $\Sigma_2$ of a graph $G$ are not automorphic to each other, then we say that they are distinct automorphic type signatures. For example, in signed graphs $[K_6, \{u_1u_2\}]$ and $[K_6, \{u_3u_5\}]$, the signatures $\{u_1u_2\}$ and $\{u_3u_5\}$ are automorphic type signatures.

The \textit{distance} between two edges $e_{1}$ and $e_{1}$ of a graph $G$, denoted by $d_{G}(e_{1},e_{2})$, is the number of vertices of a shortest path connecting their end points. For example, in the complete graph $K_6$ of Figure~\ref{1}, for edges $e_1=u_1u_2$ and $e_2=u_3u_4$, we have $d_{G}(e_{1},e_{2})=2$.

 Throughout this paper, a negative edge in a signed graph is drawn as dashed line, and a positive edge is drawn as solid line. In the next section, we discuss signings on $K_6$. 

\section{Signings on $K_6$} 
The complete graph $K_6$ is shown in Figure~\ref{1}. 
\begin{figure}[h]
\centering
\begin{tikzpicture}[scale=0.35]
\node[vertex] (v1) at (12,7) {};
\node [below] at (12,8.2) {$u_{1}$};
\node[vertex] (v2) at (15,5.1) {};
\node [below] at (15.9,5.5) {$u_{2}$};
\node[vertex] (v3) at (15,2) {};
\node [below] at (15.7,2.5) {$u_{3}$};
\node[vertex] (v4) at (12,0) {};
\node [below] at (12,-0.2) {$u_{4}$};
\node[vertex] (v5) at (9,2) {};
\node [below] at (8.2,2.7) {$u_{5}$};
\node[vertex] (v6) at (9,5.3) {};
\node [below] at (8.2,5.8) {$u_{6}$};
\foreach \from/\to in {v1/v2,v2/v3,v3/v4,v4/v5,v5/v6,v1/v6,v1/v3,v1/v4,v1/v5,v2/v4,v2/v5,v2/v6,v3/v5,v3/v6,v4/v6} \draw (\from) -- (\to);
\end{tikzpicture}
\caption{The complete graph $K_6$.}
\label{1}
\end{figure}
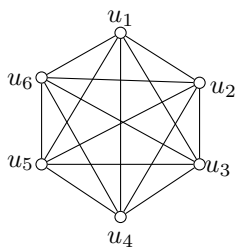
For a signed graph $[G, \Sigma]$, let the graph $G_{\Sigma}$ be such that $V(G_{\Sigma})$ = $V(G)$ and $E(G_{\Sigma})$ = $\Sigma$.
\begin{theorem}\label{Min}
Let $[G, \Sigma]$ be a signed graph on $n$ vertices and let $\Sigma^{\prime}$ be an equivalent minimal signature of $\Sigma$. Then $d_{G_{\Sigma^{\prime}}}(v)$ $\leq$ $\lfloor \frac{n-1}{2} \rfloor$ for each vertex $v \in V(G_{\Sigma^{\prime}})$.
\end{theorem}
\begin{proof}
Let, if possible, there exists a vertex $u \in V(G_{\Sigma})$ such that $d_{G_{\Sigma}}(u) > \frac{n-1}{2}$. Resign at $u$ to get an equivalent signature $\Sigma_1$. It is clear that $|\Sigma| > |\Sigma_1|$. We apply the same operation on $\Sigma_1$, if $G_{\Sigma_1}$ has a vertex of degree greater than $\frac{n-1}{2}$. Repeated application, if needed, of this process will ultimately give us an equivalent signature $\tilde{\Sigma}$ of minimum number of edges such that degree of every vertex of $\tilde{\Sigma}$ is atmost $\lfloor \frac{n-1}{2} \rfloor$. It is clear that $|\tilde{\Sigma}| = |\Sigma^{\prime}|$, and every vertex of $\Sigma^{\prime}$ have degree atmost $\lfloor \frac{n-1}{2} \rfloor$.
\end{proof}
As a particular case of Theorem~\ref{Min}, we get the following corollary.

\begin{corollary} \label{minimal}
Let $[K_6, \Sigma]$ be a signed graph and $\Sigma^{\prime}$ be an equivalent minimal signature of $\Sigma$. Then the size of $\Sigma^{\prime}$ is at most 6.
\end{corollary}

Proposition~\ref{Distinct SG} tells us that $K_6$ has $2^{10}$ distinct signed graphs. But in some respect, only 16 of them are different. Basically, we want to find the different signatures on $K_6$ upto switching isomorphism. Corollary~\ref{minimal} suggests that, it is enough to find the non-isomorphic signatures on $K_6$ of size upto six and vertex degree is atmost two. Further, $K_6$ is vertex as well as edge-transitive. We use these facts to find distinct automorphic type signatures of various sizes of $K_6$ in the following lemmas.

We denote a signature of size zero by $\Sigma_0$. We know that $K_6$ is edge-transitive, so all signatures of size one are automorphic and we denote this automorphic type signature by $\Sigma_1(u_1u_2)$.

\begin{lemm} \label{AT2}
The number of distinct automorphic type signatures of $K_6$ of size two is 2. 
\end{lemm}
\begin{proof}
For a signature of size two in $K_6$, followings are the only possibilities.
\begin{enumerate}
\item[(i)] Edges of the signature form a path of length two. One of such signatures is $\Sigma_2(u_1u_2, u_2u_3)$.
\item[(ii)] Edges of the signature are at distance two. One of such signatures is $\Sigma_3(u_1u_2, u_3u_4)$.
\end{enumerate}
 It is easy to see that, any other signature of $K_6$ of size two is either automorphic to $\Sigma_2$ or $\Sigma_3$, and that $\Sigma_2$ is not automorphic to $\Sigma_3$. This proves the lemma.
\end{proof}

\begin{lemm}\label{AT3}
The number of distinct automorphic type signatures of $K_6$ of size three is 4.
\end{lemm}
\begin{proof}
For a signature of size three in $K_6$, followings are the only possibilities.
\begin{enumerate}
\item[(i)] Edges of the signature form a path of length three. One of such signatures is $\Sigma_4(u_1u_2, u_2u_3, u_3u_4)$.
\item[(ii)] Two edges form a path and the third edge is at distance two from that path. One of such signatures is $\Sigma_5(u_1u_2, u_2u_3, u_4u_5)$.
\item[(iii)] All three edges are pairwise at distance two. One of such signatures is $\Sigma_6(u_1u_2, u_6u_3, u_5u_4)$.
\item[(iv)] The three edges of the signature form a cycle. One of such signature is $\Sigma_7(u_1u_2, u_2u_3, u_3u_1)$.
\end{enumerate}
It is clear that any other signature of size three of $K_6$ is automorphic to one of $\Sigma_4$, $\Sigma_5$, $\Sigma_6$, or $\Sigma_7$. Further, these four signatures are pairwise non-automorphic. This proves the lemma.
\end{proof}

\begin{lemm}\label{AT4}
The number of distinct automorphic type signatures of $K_6$ of size four is 5.
\end{lemm}
\begin{proof}
For a signature of size three in $K_6$, followings are the only possibilities.
\begin{enumerate}

\item[(i)] All four edges of the signature form a path. One of such signatures is $\Sigma_8(u_1u_2, u_2u_3, u_3u_4, u_4u_5)$.
\item[(ii)] Three edges of the signature form a path and the remaining edge is at distance two from this path. One of such signatures is $\Sigma_9(u_1u_2, u_2u_3, u_3u_4, u_5u_6)$.
\item[(iii)] Two edges of the signature lie on a path and other two edges lie on an another path, disjoint from the first path. One of such signatures is $\Sigma_{10}(u_1u_2, u_6u_1, u_3u_4, u_4u_5)$.
\item[(iv)] Three edges of the signature form a cycle and the remaining edge is at distance two from that cycle. One of such signatures is $\Sigma_{11}(u_1u_2, u_2u_3, u_3u_1, u_5u_6)$.
\item[(v)] All four edges of the signature form a cycle. One of such signatures is $\Sigma_{12}(u_1u_2, u_2u_3, u_3u_6, u_6u_1)$.
\end{enumerate}
It is easy to see that any other signature of size four in $K_6$ is automorphic to one of $\Sigma_{8}$, $\Sigma_{9}$, $\Sigma_{10}$, $\Sigma_{11}$ or $\Sigma_{12}$. Further, these five signatures are pairwise non-automorphic. This proves the lemma.
\end{proof}

\begin{lemm}\label{AT5}
The number of distinct automorphic type signatures of $K_6$ of  size five is 4.
\end{lemm}
\begin{proof}
It is easy to see that any subgraph of $K_6$ having five edges and having maximum degree two cannot have two disjoint paths of length 1 and 4 or 2 and 3. Therefore for a signature of size five in $K_6$, followings are the only possibilities.
\begin{enumerate}
\item[(i)] All the edges of the signature form a path of length five. One of such signatures is \linebreak[4] $\Sigma_{13}(u_1u_2, u_2u_3, u_3u_4, u_4u_5, u_5u_6)$.
\item[(ii)] Four edges of the signature form a cycle and the remaining edge is just a path of length one and disjoint from that cycle. One of such signatures is $\Sigma_{14}(u_1u_2, u_2u_3, u_3u_4, u_4u_1, u_5u_6)$.
\item[(iii)] Three edges of the signature form a cycle and the remaining two edges form a path of length two. One of such signatures is $\Sigma_{15}(u_1u_2, u_2u_3, u_3u_1, u_4u_6, u_5u_6)$.
\item[(iv)] All edges of the signature form a cycle. One of such signatures is $\Sigma_{16}(u_1u_2, u_2u_3, u_3u_4, u_4u_5, u_5u_1)$.
\end{enumerate}
It is clear that any other signature of size five of $K_6$ is automorphic to one of $\Sigma_{13}$, $\Sigma_{14}$, $\Sigma_{15}$ or $\Sigma_{16}$. Also, these four signatures are pairwise non-automorphic. This proves the lemma.
\end{proof}

\begin{lemm} \label{AT6}
The number of distinct automorphic type signatures of $K_6$ of size six is 2.
\end{lemm}
\begin{proof}
It is easy to see that a subgraph of $K_6$ having six edges and having maximum degree two is either a spanning cycle or union of two 3-cycles. One of such signatures whose edges form a spanning cycle is $\Sigma_{17}(u_1u_2, u_2u_3, u_3u_4, u_4u_5, u_5u_6, u_6u_1)$. Again, one of such signatures whose edges form two disjoint 3-cycles is $\Sigma_{18}(u_1u_2, u_2u_3, u_3u_1, u_4u_5, u_5u_6, u_6u_4)$. It is clear that any signature of size six of $K_6$ is automorphic to one of $\Sigma_{17}$ or $\Sigma_{18}$. These two signatures are non-automorphic too. This proves the lemma.
\end{proof}

The signatures obtained in the previous five lemmas along with $\Sigma_0$ and $\Sigma_1$ give us 19 distinct automorphic type signatures of $K_6$, \textit{viz.}, $\Sigma_0, \Sigma_1, \ldots, \Sigma_{18}$. Notice that any two signatures belonging to any one of $\{\Sigma_0, \Sigma_1\}$, $\{\Sigma_2, \Sigma_3\}$, $\{\Sigma_4, \Sigma_5, \Sigma_6, \Sigma_7\}$, $\{\Sigma_8, \Sigma_9, \Sigma_{10}, \Sigma_{11}, \Sigma_{12}\}$, $\{\Sigma_{13}, \Sigma_{14}, \Sigma_{15}, \Sigma_{16}\}$ or $\{\Sigma_{17}, \Sigma_{18}\}$ are not automorphic. However, two among these 19 signatures may be switching isomorphic to each other. We have the following observations.

\begin{itemize}
\item In $\Sigma_{6}$, by resigning at $\{u_{2}, u_{3}, u_{4}\}$, we get a signature which is automorphic to $\Sigma_{17}$. Thus $\Sigma_{6}$ is switching isomorphic to $\Sigma_{17}$, that is, $\Sigma_{6}$ $\sim$ $\Sigma_{17}$.
\item In $\Sigma_{10}$, by resigning at $\{u_{1}, u_{3}, u_{5}\}$, we get a signature which is auotomprphic to $\Sigma_{14}$. Thus $\Sigma_{10}$ is switching isomorphic to $\Sigma_{14}$, that is, $\Sigma_{10}$ $\sim$ $\Sigma_{14}$.
\item In $\Sigma_{13}$, by resigning at $\{u_{2}, u_{4}, u_{6}\}$, we get a signature which is automorphic to $\Sigma_{9}$. Thus $\Sigma_{9}$ is switching isomorphic to $\Sigma_{13}$, that is, $\Sigma_{9}$ $\sim$ $\Sigma_{13}$.
\end{itemize}
Thus we are left with the signatures $\Sigma_{0}$, $\Sigma_{1}$, $\Sigma_{2}$, $\Sigma_{3}$, $\Sigma_{4}$, $\Sigma_{5}$, $\Sigma_{6}$, $\Sigma_{7}$, $\Sigma_{8}$, $\Sigma_{9}$, $\Sigma_{10}$, $\Sigma_{11}$, $\Sigma_{12}$, $\Sigma_{15}$, $\Sigma_{16}$ and $\Sigma_{18}$, and their corresponding signified graphs are shown in Figure~\ref{DS62}. Now we show that no two of these 16 signatures are switching isomorphic. 

\begin{theorem}
There are exactly 16 different signatures on $K_6$ upto switching isomorphism.
\end{theorem}
\begin{proof}
Let the number of negative 3-cycles, number of negative 4-cycles and number of negative 5-cycles of a signified graph $(K_6, \Sigma)$ be denoted by $|C_3^{-}|$, $|C_4^{-}|$, and $|C_5^{-}|$, respectively. These numbers for the sixteen signed $K_6$ depicted in Figure~\ref{DS62} are given in Table~\ref{negative cycles}.

\begin{table}[h]
\begin{center}
\begin{tabular}{|c|c|c|c|c|c|c|c|c|c|c|c|c|c|c|c|c|}
\hline
  & $\Sigma_{0}$ & $\Sigma_{1}$ & $\Sigma_{2}$ & $\Sigma_{3}$ & $\Sigma_{4}$ & $\Sigma_{5}$ & $\Sigma_{6}$ & $\Sigma_{7}$ & $\Sigma_{8}$ & $\Sigma_{9}$ & $\Sigma_{10}$ & $\Sigma_{11}$ & $\Sigma_{12}$ & $\Sigma_{15}$ & $\Sigma_{16}$ & $\Sigma_{18}$\\
\hline
$|C_{3}^{-}|$ & 0 & 4 & 6 & 8 & 8 & 10 & 12 & 10 & 10 & 12 & 12 & 14 & 8 & 16 & 10 & 20\\
\hline
$|C_{4}^{-}|$ & 0 & 12 & 18 & 20 & 24 & 22 & 24 & 18 & 26 & 24 & 20 & 18 & 24 & 12 & 30 & 0\\
\hline
$|C_{5}^{-}|$ & 0 & 24 & 24 & 32 & 40 & 36 & 24 & 36 & 36 & 32 & 40 & 36 & 48 & 48 & 36 & 72\\
\hline

\end{tabular}
\end{center}
\caption{Number of negative 3-cycles, 4-cycles and 5-cycles in different signed $K_6$.}
\label{negative cycles}
\end{table}

From Table~\ref{negative cycles} and Theorem~\ref{Signature}, it is easy to see that all the signatures depicted in Figure~\ref{DS62} are pairwise non-isomorphic.
This concludes the proof of the theorem.
\end{proof}

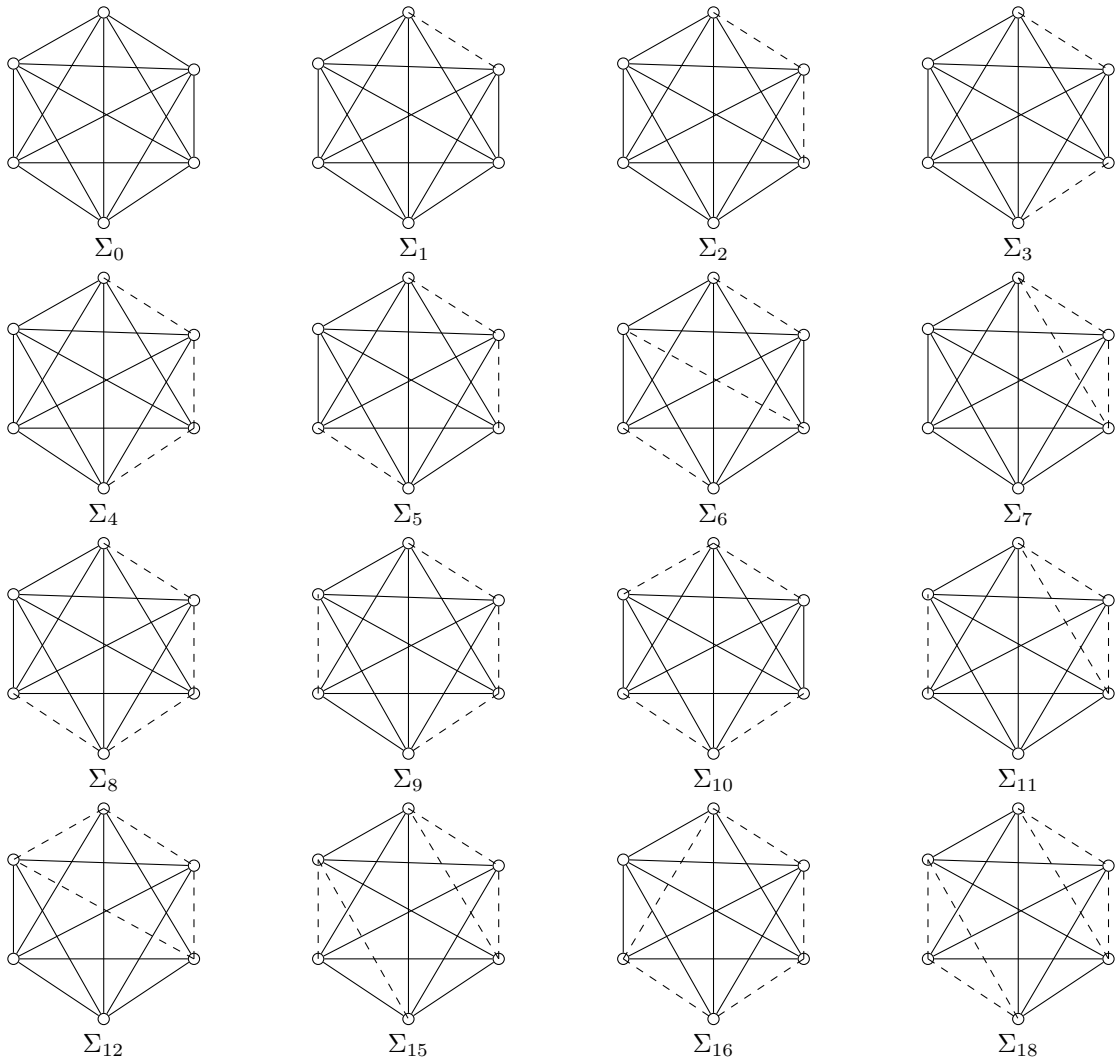
\begin{figure}[h]
\begin{subfigure}{0.24\textwidth}
\begin{tikzpicture}[scale=0.4]
\node[vertex] (v1) at (12,7) {};
\node[vertex] (v2) at (15,5.1) {};
\node[vertex] (v3) at (15,2) {};
\node[vertex] (v4) at (12,0) {};
\node [below] at (12.2,-0.2) {$\Sigma_{0}$};
\node[vertex] (v5) at (9,2) {};
\node[vertex] (v6) at (9,5.3) {};

\foreach \from/\to in {v1/v2,v2/v3,v3/v4,v4/v5,v5/v6,v1/v6,v1/v3,v1/v4,v1/v5,v2/v4,v2/v5,v2/v6,v3/v5,v3/v6,v4/v6} \draw (\from) -- (\to);

\end{tikzpicture}
\end{subfigure}
\hfill
\begin{subfigure}{0.24\textwidth}
\begin{tikzpicture}[scale=0.4]
\node[vertex] (v1) at (12,7) {};
\node[vertex] (v2) at (15,5.1) {};
\node[vertex] (v3) at (15,2) {};
\node[vertex] (v4) at (12,0) {};
\node [below] at (12.2,-0.2) {$\Sigma_{1}$};
\node[vertex] (v5) at (9,2) {};
\node[vertex] (v6) at (9,5.3) {};

\foreach \from/\to in {v2/v3,v3/v4,v4/v5,v5/v6,v1/v6,v1/v3,v1/v4,v1/v5,v2/v4,v2/v5,v2/v6,v3/v5,v3/v6,v4/v6} \draw (\from) -- (\to);

\draw [dashed] (12,7) -- (15,5.1);

\end{tikzpicture}
\end{subfigure}
\hfill
\begin{subfigure}{0.24\textwidth}
\begin{tikzpicture}[scale=0.4]
\node[vertex] (v1) at (12,7) {};
\node[vertex] (v2) at (15,5.1) {};
\node[vertex] (v3) at (15,2) {};
\node[vertex] (v4) at (12,0) {};
\node [below] at (12,-0.2) {$\Sigma_{2}$};
\node[vertex] (v5) at (9,2) {};
\node[vertex] (v6) at (9,5.3) {};

\foreach \from/\to in {v3/v4,v4/v5,v5/v6,v1/v6,v1/v3,v1/v4,v1/v5,v2/v4,v2/v5,v2/v6,v3/v5,v3/v6,v4/v6} \draw (\from) -- (\to);

\draw [dashed] (12,7) -- (15,5.1);
\draw [dashed] (15,2) -- (15,5.1);

\end{tikzpicture}
\end{subfigure}
\hfill
\begin{subfigure}{0.24\textwidth}
\begin{tikzpicture}[scale=0.4]
\node[vertex] (v1) at (12,7) {};
\node[vertex] (v2) at (15,5.1) {};
\node[vertex] (v3) at (15,2) {};
\node[vertex] (v4) at (12,0) {};
\node [below] at (12,-0.2) {$\Sigma_{3}$};
\node[vertex] (v5) at (9,2) {};
\node[vertex] (v6) at (9,5.3) {};

\foreach \from/\to in {v2/v3,v4/v5,v5/v6,v1/v6,v1/v3,v1/v4,v1/v5,v2/v4,v2/v5,v2/v6,v3/v5,v3/v6,v4/v6} \draw (\from) -- (\to);

\draw [dashed] (12,7) -- (15,5.1);
\draw [dashed] (15,2) -- (12,0);

\end{tikzpicture}
\end{subfigure}
\hfill
\begin{subfigure}{0.24\textwidth}
\begin{tikzpicture}[scale=0.4]
\node[vertex] (v1) at (12,7) {};
\node[vertex] (v2) at (15,5.1) {};
\node[vertex] (v3) at (15,2) {};
\node[vertex] (v4) at (12,0) {};
\node [below] at (12,-0.2) {$\Sigma_{4}$};
\node[vertex] (v5) at (9,2) {};
\node[vertex] (v6) at (9,5.3) {};

\foreach \from/\to in {v4/v5,v5/v6,v1/v6,v1/v3,v1/v4,v1/v5,v2/v4,v2/v5,v2/v6,v3/v5,v3/v6,v4/v6} \draw (\from) -- (\to);

\draw [dashed] (12,7) -- (15,5.1);
\draw [dashed] (15,2) -- (12,0);
\draw [dashed] (15,2) -- (15,5.1);

\end{tikzpicture}
\end{subfigure}
\hfill
\begin{subfigure}{0.24\textwidth}
\begin{tikzpicture}[scale=0.4]
\node[vertex] (v1) at (12,7) {};
\node[vertex] (v2) at (15,5.1) {};
\node[vertex] (v3) at (15,2) {};
\node[vertex] (v4) at (12,0) {};
\node [below] at (12,-0.2) {$\Sigma_{5}$};
\node[vertex] (v5) at (9,2) {};
\node[vertex] (v6) at (9,5.3) {};

\foreach \from/\to in {v3/v4,v5/v6,v1/v6,v1/v3,v1/v4,v1/v5,v2/v4,v2/v5,v2/v6,v3/v5,v3/v6,v4/v6} \draw (\from) -- (\to);

\draw [dashed] (12,7) -- (15,5.1);
\draw [dashed] (9,2) -- (12,0);
\draw [dashed] (15,2) -- (15,5.1);

\end{tikzpicture}
\end{subfigure}
\hfill
\begin{subfigure}{0.24\textwidth}
\begin{tikzpicture}[scale=0.4]
\node[vertex] (v1) at (12,7) {};
\node[vertex] (v2) at (15,5.1) {};
\node[vertex] (v3) at (15,2) {};
\node[vertex] (v4) at (12,0) {};
\node [below] at (12,-0.2) {$\Sigma_{6}$};
\node[vertex] (v5) at (9,2) {};
\node[vertex] (v6) at (9,5.3) {};

\foreach \from/\to in {v2/v3,v3/v4,v5/v6,v1/v6,v1/v3,v1/v4,v1/v5,v2/v4,v2/v5,v2/v6,v3/v5,v4/v6} \draw (\from) -- (\to);

\draw [dashed] (12,7) -- (15,5.1);
\draw [dashed] (9,2) -- (12,0);
\draw [dashed] (15,2) -- (9,5.3);

\end{tikzpicture}
\end{subfigure}
\hfill
\begin{subfigure}{0.24\textwidth}
\begin{tikzpicture}[scale=0.4]
\node[vertex] (v1) at (12,7) {};
\node[vertex] (v2) at (15,5.1) {};
\node[vertex] (v3) at (15,2) {};
\node[vertex] (v4) at (12,0) {};
\node [below] at (12,-0.2) {$\Sigma_{7}$};
\node[vertex] (v5) at (9,2) {};
\node[vertex] (v6) at (9,5.3) {};

\foreach \from/\to in {v3/v4,v4/v5,v5/v6,v1/v6,v1/v4,v1/v5,v2/v4,v2/v5,v2/v6,v3/v5,v3/v6,v4/v6} \draw (\from) -- (\to);

\draw [dashed] (12,7) -- (15,5.1);
\draw [dashed] (12,7) -- (15,2);
\draw [dashed] (15,2) -- (15,5.1);

\end{tikzpicture}
\end{subfigure}
\hfill
\begin{subfigure}{0.24\textwidth}
\begin{tikzpicture}[scale=0.4]
\node[vertex] (v1) at (12,7) {};
\node[vertex] (v2) at (15,5.1) {};
\node[vertex] (v3) at (15,2) {};
\node[vertex] (v4) at (12,0) {};
\node [below] at (12,-0.2) {$\Sigma_{8}$};
\node[vertex] (v5) at (9,2) {};
\node[vertex] (v6) at (9,5.3) {};

\foreach \from/\to in {v5/v6,v1/v6,v1/v3,v1/v4,v1/v5,v2/v4,v2/v5,v2/v6,v3/v5,v3/v6,v4/v6} \draw (\from) -- (\to);

\draw [dashed] (12,7) -- (15,5.1);
\draw [dashed] (9,2) -- (12,0);
\draw [dashed] (15,2) -- (15,5.1);
\draw [dashed] (15,2) -- (12,0);

\end{tikzpicture}
\end{subfigure}
\hfill
\begin{subfigure}{0.24\textwidth}
\begin{tikzpicture}[scale=0.4]
\node[vertex] (v1) at (12,7) {};
\node[vertex] (v2) at (15,5.1) {};
\node[vertex] (v3) at (15,2) {};
\node[vertex] (v4) at (12,0) {};
\node [below] at (12,-0.2) {$\Sigma_{9}$};
\node[vertex] (v5) at (9,2) {};
\node[vertex] (v6) at (9,5.3) {};

\foreach \from/\to in {v4/v5,v1/v6,v1/v3,v1/v4,v1/v5,v2/v4,v2/v5,v2/v6,v3/v5,v3/v6,v4/v6} \draw (\from) -- (\to);

\draw [dashed] (12,7) -- (15,5.1);
\draw [dashed] (9,2) -- (9,5.3);
\draw [dashed] (15,2) -- (15,5.1);
\draw [dashed] (15,2) -- (12,0);

\end{tikzpicture}
\end{subfigure}
\hfill
\begin{subfigure}{0.24\textwidth}
\begin{tikzpicture}[scale=0.4]
\node[vertex] (v1) at (12,7) {};
\node[vertex] (v2) at (15,5.1) {};
\node[vertex] (v3) at (15,2) {};
\node[vertex] (v4) at (12,0) {};
\node [below] at (12,-0.2) {$\Sigma_{10}$};
\node[vertex] (v5) at (9,2) {};
\node[vertex] (v6) at (9,5.3) {};

\foreach \from/\to in {v2/v3,v5/v6,v1/v3,v1/v4,v1/v5,v2/v4,v2/v5,v2/v6,v3/v5,v3/v6,v4/v6} \draw (\from) -- (\to);

\draw [dashed] (12,7) -- (15,5.1);
\draw [dashed] (12,7) -- (9,5.3);
\draw [dashed] (15,2) -- (12,0);
\draw [dashed] (9,2) -- (12,0);

\end{tikzpicture}
\end{subfigure}
\hfill
\begin{subfigure}{0.24\textwidth}
\begin{tikzpicture}[scale=0.4]
\node[vertex] (v1) at (12,7) {};
\node[vertex] (v2) at (15,5.1) {};
\node[vertex] (v3) at (15,2) {};
\node[vertex] (v4) at (12,0) {};
\node [below] at (12,-0.2) {$\Sigma_{11}$};
\node[vertex] (v5) at (9,2) {};
\node[vertex] (v6) at (9,5.3) {};

\foreach \from/\to in {v3/v4,v4/v5,v1/v6,v1/v4,v1/v5,v2/v4,v2/v5,v2/v6,v3/v5,v3/v6,v4/v6} \draw (\from) -- (\to);

\draw [dashed] (12,7) -- (15,5.1);
\draw [dashed] (9,2) -- (9,5.3);
\draw [dashed] (15,2) -- (15,5.1);
\draw [dashed] (15,2) -- (12,7);

\end{tikzpicture}
\end{subfigure}
\hfill
\begin{subfigure}{0.24\textwidth}
\begin{tikzpicture}[scale=0.4]
\node[vertex] (v1) at (12,7) {};
\node[vertex] (v2) at (15,5.1) {};
\node[vertex] (v3) at (15,2) {};
\node[vertex] (v4) at (12,0) {};
\node [below] at (12,-0.2) {$\Sigma_{12}$};
\node[vertex] (v5) at (9,2) {};
\node[vertex] (v6) at (9,5.3) {};

\foreach \from/\to in {v3/v4,v4/v5,v5/v6,v1/v3,v1/v4,v1/v5,v2/v4,v2/v5,v2/v6,v3/v5,v4/v6} \draw (\from) -- (\to);

\draw [dashed] (12,7) -- (15,5.1);
\draw [dashed] (12,7) -- (9,5.3);
\draw [dashed] (15,2) -- (15,5.1);
\draw [dashed] (15,2) -- (9,5.3);

\end{tikzpicture}
\end{subfigure}
\hfill
\begin{subfigure}{0.24\textwidth}
\begin{tikzpicture}[scale=0.4]
\node[vertex] (v1) at (12,7) {};
\node[vertex] (v2) at (15,5.1) {};
\node[vertex] (v3) at (15,2) {};
\node[vertex] (v4) at (12,0) {};
\node [below] at (12,-0.2) {$\Sigma_{15}$};
\node[vertex] (v5) at (9,2) {};
\node[vertex] (v6) at (9,5.3) {};

\foreach \from/\to in {v4/v5,v1/v6,v1/v4,v1/v5,v2/v4,v2/v5,v2/v6,v3/v5,v3/v6,v3/v4} \draw (\from) -- (\to);

\draw [dashed] (12,7) -- (15,5.1);
\draw [dashed] (9,2) -- (9,5.3);
\draw [dashed] (15,2) -- (15,5.1);
\draw [dashed] (15,2) -- (12,7);
\draw [dashed] (9,5.3) -- (12,0);

\end{tikzpicture}
\end{subfigure}
\hfill
\begin{subfigure}{0.24\textwidth}
\begin{tikzpicture}[scale=0.4]
\node[vertex] (v1) at (12,7) {};
\node[vertex] (v2) at (15,5.1) {};
\node[vertex] (v3) at (15,2) {};
\node[vertex] (v4) at (12,0) {};
\node [below] at (12,-0.2) {$\Sigma_{16}$};
\node[vertex] (v5) at (9,2) {};
\node[vertex] (v6) at (9,5.3) {};

\foreach \from/\to in {v1/v6,v1/v3,v1/v4,v6/v5,v2/v4,v2/v5,v2/v6,v3/v5,v3/v6,v4/v6} \draw (\from) -- (\to);

\draw [dashed] (12,7) -- (15,5.1);
\draw [dashed] (15,2) -- (15,5.1);
\draw [dashed] (15,2) -- (12,0);
\draw [dashed] (9,2) -- (12,0);
\draw [dashed] (9,2) -- (12,7);

\end{tikzpicture}
\end{subfigure}
\hfill
\begin{subfigure}{0.24\textwidth}
\begin{tikzpicture}[scale=0.4]
\node[vertex] (v1) at (12,7) {};
\node[vertex] (v2) at (15,5.1) {};
\node[vertex] (v3) at (15,2) {};
\node[vertex] (v4) at (12,0) {};
\node [below] at (12,-0.2) {$\Sigma_{18}$};
\node[vertex] (v5) at (9,2) {};
\node[vertex] (v6) at (9,5.3) {};

\foreach \from/\to in {v1/v6,v1/v4,v1/v5,v2/v4,v2/v5,v2/v6,v3/v5,v3/v6,v3/v4} \draw (\from) -- (\to);

\draw [dashed] (12,7) -- (15,5.1);
\draw [dashed] (9,2) -- (9,5.3);
\draw [dashed] (15,2) -- (15,5.1);
\draw [dashed] (15,2) -- (12,7);
\draw [dashed] (9,5.3) -- (12,0);
\draw [dashed] (9,2) -- (12,0);

\end{tikzpicture}
\end{subfigure}
\caption{The sixteen signed $K_6$.}
\label{DS62}
\end{figure}

\section{Conclusions and Remarks}
We described the different signatures of complete graph $K_6$ upto switching isomorphism. In \cite{Zaslavsky1}, Zaslavsky introduced the concepts of signed graph colouring and signed chromatic number of signed graphs. Zaslavsky also shown that these parameters are invariant under switching. So, study of these invariants for the sixteen non-isomorphic signatures of $K_6$ would be interesting.

\end{document}